\documentclass[12pt]{amsart}
\usepackage{ifthen,verbatim}
\usepackage{floatflt,graphicx,graphics}
\usepackage{subfig}

 \usepackage{tikz} 
\usepackage{mathrsfs}
\usepackage{amsmath,amssymb,amsthm}

\numberwithin{equation}{section}
\nonstopmode
\numberwithin{equation}{section}
\theoremstyle{plain}
\newtheorem{theorem}[equation]{Theorem}
\newtheorem{conjecture}[equation]{Conjecture}
\newtheorem{lemma}[equation]{Lemma}
\newtheorem{corollary}[equation]{Corollary}

\theoremstyle{definition}

\newtheorem{remark}[equation]{Remark}
\newtheorem{nonsec}[equation]{}

\theoremstyle{remark}

\newcommand{\R}{\mathbb{R}}

\newcommand{\B}{\mathbb{B}}

{\qed\bigskip}

\newcounter{alphabet}

\newcounter{minutes}
\divide\time by 60
\newcounter{hours}
\multiply\time by 60
\addtocounter{minutes}{-\time}

\begin{document}
\bibliographystyle{amsplain}
\title
{
Inclusion properties of the triangular ratio metric balls
}

\def\thefootnote{}
\footnotetext{
\texttt{\tiny File:~\jobname .tex,
          printed: \number\year-\number\month-\number\day,
          \thehours.\ifnum\theminutes<10{0}\fi\theminutes}
}
\makeatletter\def\thefootnote{\@arabic\c@footnote}\makeatother

\author[O. Rainio]{Oona Rainio}

\begin{abstract}
Inclusion properties are studied for balls of the triangular ratio metric, the hyperbolic metric, the $j^*$-metric, and the distance ratio metric defined in the unit ball domain. Several sharp results are proven and a conjecture about the relation between triangular ratio metric balls and hyperbolic balls is given. An algorithm is also built for drawing triangular ratio circles or three-dimensional spheres.
\end{abstract}
\maketitle

\noindent\textbf{Author information.}\\ 
email: \texttt{ormrai@utu.fi},\\
ORCID: 0000-0002-7775-7656,\\
affiliation: University of Turku, FI-20014 Turku, Finland\\
\textbf{Availability of data and material.} Not applicable, no new data was generated.\\
\textbf{Competing interests.} There are no competing interests.\\
\textbf{Funding.} My research was funded by the Finnish Culture Foundation.\\
\textbf{Acknowledgments.} 
I wish to thank Prof. Matti Vuorinen for suggesting this topic to me and for all his help and support. 

\section{Introduction}

In the field of geometric function theory, the study focuses on different geometrical properties of several types of functions, including for instance conformal, quasiconformal, and quasiregular mappings. It is often useful to construct a metric space to investigate behaviour of metrics under these mappings to understand better how they distort distances. Especially, the hyperbolic metric is an important subject of study because of its conformal invariance, sensitivity to boundary variation, monotonicity with respect to domain, and intricate features of hyperbolic geometry \cite[p. 191-192]{hkv}. Numerous generalizations of the hyperbolic metric called hyperbolic type metrics have also been defined to create geometrical systems sharing some of these properties, including the way the hyperbolic metric measures distances between two points $x,y$ in a domain $G$ by taking into account their position with respect to the boundary $\partial G$ of the domain.

For a domain $G\subsetneq\R^n$, one of the hyperbolic type metrics is the \emph{triangular ratio metric} $s_G:G\times G\to[0,1],$  \cite[(1.1), p. 683]{chkv}
\begin{align}
s_G(x,y)=\frac{|x-y|}{\inf_{z\in\partial G}(|x-z|+|z-y|)}, 
\end{align}
which was originally introduced by P. H\"ast\"o in 2002 \cite{h}. This metric has been studied much \cite{chkv, fhmv, hkvz, ird, sch, fss, sinb, sqm} because, despite its relatively simple definition, computing the value of the infimum in the denominator is a non-trivial task in the case where the domain $G$ is, for instance, the unit ball $\B^n$. In fact, if $G=\B^3$ and the boundary of this domain is a spherical mirror, the correct point $z$ coincides with the point that the light ray sent from the point $x$ must hit to reflect into the point $y$ and, because of this relation to optics, the problem has a very long history explained in \cite{fhmv}.

Given the point $x$ in a domain $G\subsetneq\R^n$ and some radius $0<t<1$, the $x$-centered triangular ratio metric ball with this radius $t$ is defined as $B_s(x,t)=\{y\in G\text{ }|\text{ }s_G(x,y)<t\}$. To study the different mappings, it is often useful to know what is the largest ball $B_d(x,r_0)$ defined with some other metric $d$ that is included in the triangular ratio metric ball $B_s(x,t)$ or the smallest ball $B_d(x,r_1)$ containing the ball $B_s(x,t)$. This type of an inclusion question has been recently studied for the triangular metric by S. Hokuni, R. Klén, Y. Li, and M. Vuorinen \cite{h16} in such cases where the domain $G$ is the upper half-space or the punctured real space, and Klén and Vuorinen have also researched the inclusion properties of several hyperbolic metrics other than the triangular ratio metric in the unit ball \cite{k13} and other domains \cite{k12}. 

In this article we continue the earlier research by studying the inclusion properties of the balls of the triangular ratio metric defined in the unit ball. In Section 3, we focus on triangular ratio metric balls in terms of the Euclidean geometry, give the sharp inclusion result between triangular ratio metric and Euclidean balls, and also offer an algorithm for drawing triangular ratio metric circles or three-dimensional spheres. In Section 4, we study the ball inclusion for the triangular ratio metric and two other hyperbolic type metrics called the distance ratio metric and the $j^*$-metric, find an explicit formula for $j^*$-metric balls and also prove one earlier conjecture from \cite{h16}. Finally, in Section 5, we research the connection between triangular ratio metric balls and hyperbolic balls.

\section{Preliminaries}

An $n$-dimensional Euclidean ball with center $x\in\R^n$ and radius $r>0$ is denoted by $B^n(x,r)$, its closure by $\overline{B}^n(x,r)$ and its sphere by $S^{n-1}(x,r)$. For the unit ball and sphere, the simplified notations $\B^n$ and $S^{n-1}$ are used. The notation $L(x,y)$ means a Euclidean line through points $x,y\in\R^n$. For any point $x$ in a domain $G\subsetneq\R^n$, $d_G(x)$ denotes the Euclidean distance from $x$ to the boundary $\partial G$ by $d_G(x)=\inf\{|x-z|\text{ }|\text{ }z\in\partial G\}$. Trivially $d_G(x)=1-|x|$ when $G=\B^n$. For $x\in\R^2$, let $\overline{x}$ be the complex conjugate of $x$. For three points $x,y,z\in\R^n$, let $\measuredangle XYZ\in[0,\pi]$ be the magnitude of the angle between the vector from $y$ to $x$ and the vector from $y$ to $z$ and denote the origin with the letter $o$ in these angle notations.

The \emph{hyperbolic metric} can be computed in the unit ball with the following formula \cite[(4.16), p. 55]{hkv}
\begin{align}\label{def_rho}
\text{sh}^2\frac{\rho_{\B^n}(x,y)}{2}&=\frac{|x-y|^2}{(1-|x|^2)(1-|y|^2)},\quad x,y\in\B^n,
\end{align}
the \emph{distance ratio metric} $j_G:G\times G\to[0,\infty)$ in any domain $G\subsetneq\R^n$ is defined as \cite[p. 685]{chkv},
\begin{align}
j_G(x,y)=\log\left(1+\frac{|x-y|}{\min\{d_G(x),d_G(y)\}}\right),  
\end{align}
and the \emph{$j^*$-metric} $j^*_G:G\times G\to[0,1],$ has the definition \cite[2.2, p. 1123 \& Lemma 2.1, p. 1124]{hvz}
\begin{align}\label{def_j}
j^*_G(x,y)={\rm th}\frac{j_G(x,y)}{2}=\frac{|x-y|}{|x-y|+2\min\{d_G(x),d_G(y)\}},    
\end{align}
In \eqref{def_rho}, sh denotes the hyperbolic sine. Similarly, the hyperbolic cosine and tangent are denoted here by ch and th, and the inverse hyperbolic functions are arsh, arch, and arch, while ${\rm acos}$ is the arccosine function. For any hyperbolic type metric $d$ defined in a domain $G\subsetneq\R^n$, $x\in G$, and $r>0$, the $n$-dimensional ball is $B_d(x,r)=\{y\in G\text{ }|\text{ }d(x,y)<r\}$ and its sphere is $S_d(x,r)$.

\begin{lemma}\label{lem_rhoB}\emph{\cite[(4.20) p. 56]{hkv}}
For $G=\B^n$, the equality $B_\rho(x,R)=B^n(y,h)$ holds, if 
\begin{align*}
y=\frac{x(1-k^2)}{1-|x|^2k^2},\quad\text{and}\quad
h=\frac{(1-|x|^2)k}{1-|x|^2k^2}
\quad\text{and}\quad
k={\rm th}\frac{R}{2}.
\end{align*}
\end{lemma}

The inclusion result between the hyperbolic and the distance ratio metric is known in the case of the unit ball.

\begin{theorem}\label{thm_distRhoInc}\emph{\cite[Thm 3.1, p. 31 \& Cor. 3.3, p. 33]{k13}}
For $x\in G=\B^n$ and $R>0$, $B_j(x,K_0)\subseteq B_\rho(x,R)\subseteq B_j(x,K_1)$ if and only if
\begin{align*}
&0< K_0\leq\max\left\{\log\left(1+(1+|x|){\rm sh}\left(\frac{R}{2}\right)\right), \log\left(1+(1-|x|)\frac{e^R-1}{2}\right)\right\}
\quad\text{and}\\
&K_1\geq\log\left(1+(1+|x|)\frac{e^R-1}{2}\right),
\end{align*}
and, for $x\in G=\B^n$ and $K>0$, $B_\rho(x,R_0)\subseteq B_\rho(x,K)\subseteq B_\rho(x,R_1)$ if and only if
\begin{align*}
0<R_0\leq\log\left(1+\frac{2(e^K-1)}{1+|x|}\right)
\quad\text{and}\quad
R_1\geq\min\left\{2{\rm arsh}\frac{e^K-1}{1+|x|},\log\left(\frac{2e^K-1-|x|}{1-|x|}\right)\right\}.
\end{align*}
\end{theorem}

The following inequality holds between the triangular ratio metric and the $j^*$-metric.

\begin{theorem}\label{thm_sjbounds}
\emph{\cite[Lemma 2.1, p. 1124; Lemma 2.2, p. 1125; Lemma 2.8 \& Thm 2.9(1), p. 1129]{hvz}} For a domain $G\subsetneq\R^n$, the inequality $j^*_G(x,y)\leq s_G(x,y)\leq2j^*_G(x,y)$ holds for all $x,y\in G$ and, if $G$ is convex, then $s_G(x,y)\leq\sqrt{2}j^*_G(x,y)$.
\end{theorem}

There are a few results that help to find the triangular ratio distance in the unit ball.

\begin{theorem}\label{thm_findzinB}
\emph{\cite[p. 138]{fhmv}} For all $x,y\in\B^n$, the infimum $\inf_{z\in S^{n-1}}(|x-z|+|z-y|)$ is found with such a point $z\in S^{n-1}$ that the line $L(0,z)$ bisects the angle between the lines $L(x,z)$ and $L(z,y)$.
\end{theorem}

\begin{lemma}\label{lem_sCollinear}
\emph{\cite[11.2.1(1) p. 205]{hkv}}
For all $x,y\in\B^n$,
\begin{align*}
s_{\B^n}(x,y)\leq\frac{|x-y|}{2-|x+y|},    
\end{align*}
where the equality holds if the points $x,y$ are collinear with the origin.
\end{lemma}

\begin{theorem}\label{thm_sConjugate}
\emph{\cite[Thm 3.1, p. 276]{hkvz}} If $x\in\B^2$ with $\min\{{\rm Re}(x),{\rm Im}(x)\}>0$, then
\begin{align*}
s_{\B^2}(x,\overline{x})&=|x|\text{ if }|x-\frac{1}{2}|>\frac{1}{2},\\
s_{\B^2}(x,\overline{x})&=\frac{{\rm Im}(x)}{\sqrt{(1-{\rm Re}(x))^2+{\rm Im}(x)^2}}\leq|x|\text{ otherwise.}
\end{align*}
\end{theorem}

\section{Triangular ratio metric balls in Euclidean geometry}

If triangular metric balls defined in $G=\B^n$ are origin-centered, they are equivalent to Euclidean balls, as stated by the following lemma. However, if $x\neq0$, expressing the triangular ratio metric ball $B_s(x,t)$ for $0<t<1$ in terms of Euclidean geometry becomes considerably more difficult, but with the help Lemma \ref{lem_sCollinear} and Theorem \ref{thm_sConjugate}, we can prove Lemmas \ref{lem_colliIntS} and \ref{lem_conjugS} that reveal certain points always included in the sphere $S_s(x,t)$.

\begin{lemma}\label{lem_sForx0}
For $G=\B^n$ and $0<t<1$, the origin-centered triangular ratio metric ball $B_s(0,s)$ equals the Euclidean ball $B^n(0,2t/(1+t))$.
\end{lemma}
\begin{proof}
Trivially, $y\in B_s(0,s)$ if and only if $s_{\B^n}(y,0)=|y|/(2-|y|)<t$ or, equivalently, $|y|<2t/(1+t)$.
\end{proof}

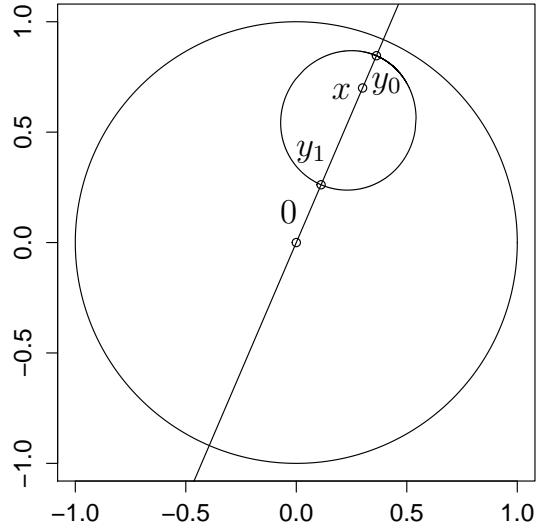
\begin{figure}[ht]
    \centering
    \begin{tikzpicture}
    \node[inner sep=0pt] at (0,0)
    {\includegraphics[width=7cm,trim={1cm 2.5cm 1cm 3cm},clip]{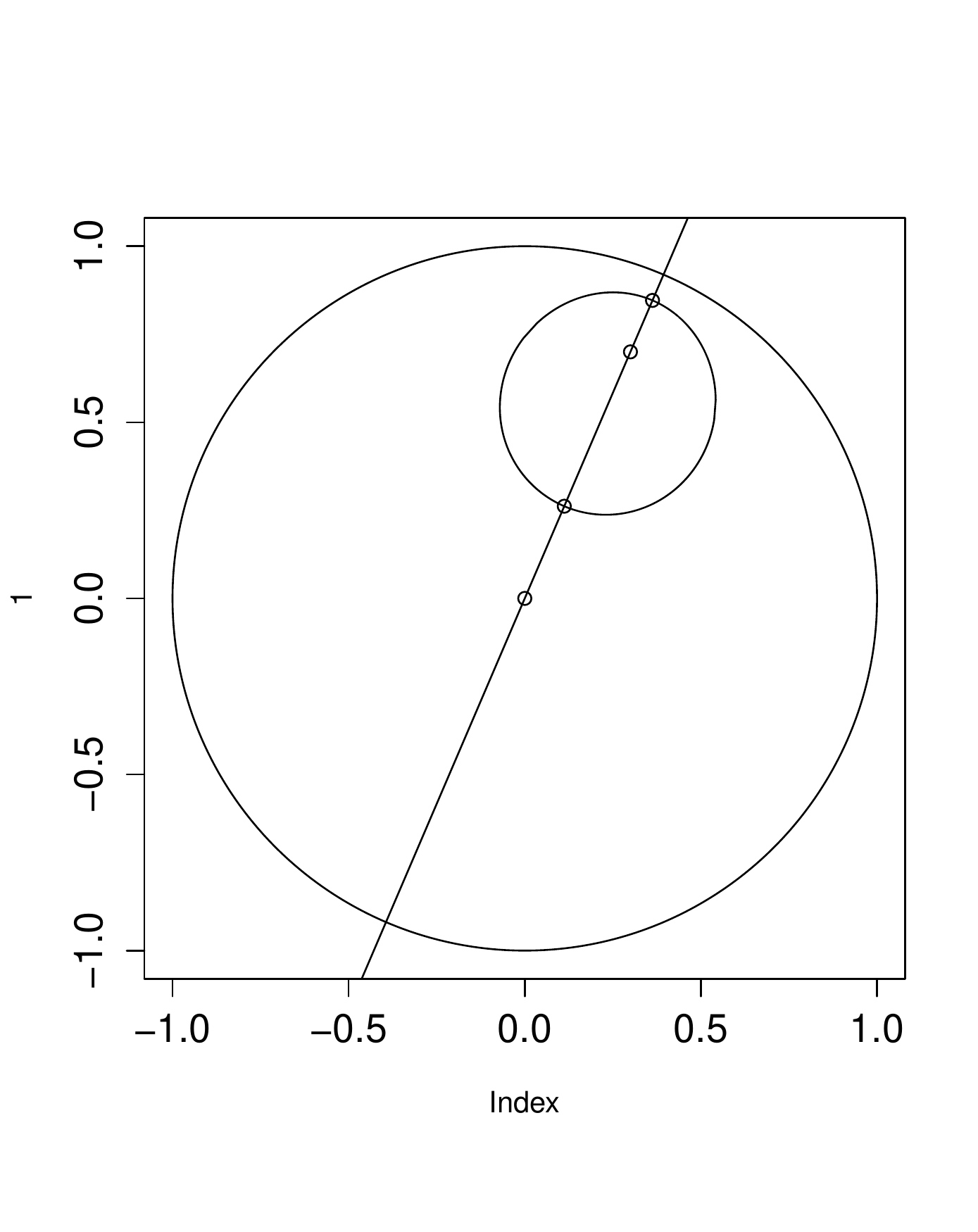}};
    \node[scale=1.1] at (1.5,2.4) {$y_0$};
    \node[scale=1.1] at (0.9,2.3) {$x$};
    \node[scale=1.1] at (0.5,1.5) {$y_1$};
    \node[scale=1.1] at (0.2,0.7) {$0$};
    \end{tikzpicture}
    \caption{The intersection points $y_0$ and $y_1$ of Lemma \ref{lem_colliIntS} for the triangular ratio metric circle $S_s(x,t)$ and the line $L(0,x)$ in the domain $G=\B^2$, when $x=0.3+0.7i$ and $t=0.5$.}
    \label{fig1}
\end{figure}

Next, we will find the intersection points of the triangular ratio metric sphere $S_s(x,t)$ and the line $L(0,x)$ that can be also seen in Figure \ref{fig1}.

\begin{lemma}\label{lem_colliIntS}
For $G=\B^n$, $x\in\B^n\setminus\{0\}$ and $0<t<1$, the triangular ratio metric sphere $S_s(x,t)$ intersects with the line $L(0,x)$ at the points
\begin{align*}
y_0=x+\frac{2tx(1-|x|)}{|x|(1+t)}\quad\text{and}\quad
y_1=x-\frac{2tx}{|x|}\min\left\{\frac{1-|x|}{1-t},\frac{1+|x|}{1+t}\right\}.
\end{align*}
\end{lemma}
\begin{proof}
Any point $y\in L(0,x)$ can be written as $y=x+rx/|x|$ with $r\in\R$. Since the points $x,y,0$ are collinear, by Lemma \ref{lem_sCollinear},
\begin{align*}
s_{\B^n}(x,y)=\frac{|x-y|}{2-|x+y|}=\frac{|r|}{2-|2|x|+r|}=
\begin{cases}
-|r|/(2+2|x|-|r|)\text{ if }r<-2|x|,\\
|r|/(2-2|x|+|r|)\text{ if }-2|x|\leq r<0,\text{ or}\\
|r|/(2-2|x|-|r|)\text{ if }r\geq0.
\end{cases}
\end{align*}
For $s_{\B^n}(x,y)=t$, we can solve
\begin{align*}
r=\begin{cases}
-2t(1+|x|)/(1+t)\text{ if }t>|x|\text{ and }r<0,\\
-2t(1-|x|)/(1-t)\text{ if }t\leq|x|\text{ and }r<0,\text{ or}\\
2t(1-|x|)/(1+t)\text{ if }r>0.
\end{cases}    
\end{align*}
Clearly, one intersection point is found with $r<0$ and the other one with $r>0$ and, since $(1-|x|)/(1-t)\leq(1+|x|)/(1+t)$ is equivalent to $t\leq|x|$, $r$ can be written as a minimum expression so that substituting these values of $r$ in $y=x+rx/|x|$ proves the lemma. 
\end{proof}

\begin{lemma}\label{lem_conjugS}
For $G=\B^2$, $x\in\B^2\setminus\{0\}$ and $0<t\leq|x|$, the triangular ratio metric circle $S_s(x,t)$ and the Euclidean circle $S^1(0,|x|)$ intersect at points 
\begin{align*}
x(2c^2-1+2c\sqrt{1-c^2}i)^{\pm1}\quad\text{with}\quad
c=\frac{t^2+\sqrt{(1-t^2)(|x|^2-t^2)}}{|x|},
\end{align*}
and, if $t>|x|$, the triangular ratio metric disk $B_s(x,t)$ contains the Euclidean circle $S^1(0,|x|)$.
\end{lemma}
\begin{proof}
Let $x=|x|e^{ki}$ and $y=|x|e^{(k\pm2u)i}$ with $0\leq k<2\pi$ $0<u<\pi/2$. By Theorem \ref{thm_sConjugate} and the fact that triangular ratio metric is invariant under rotations around the origin,
\begin{align*}
s_{\B^2}(x,y)=\frac{|x|\sin(u)}{\sqrt{(1-|x|\cos(u))^2+|x|^2\sin(u)}}\leq|x|\quad\text{if}\quad||x|e^{ui}-\frac{1}{2}|\leq\frac{1}{2}.    
\end{align*}
Since $||x|e^{ui}-1/2|=\sqrt{|x|^2-|x|\cos(u)+1/4}$, the condition above holds if and only if $\cos(u)\geq|x|$ and, since $0<u<\pi/2$, we need to have $|x|\leq\cos(u)<1$. 

By the quadratic formula, we can solve that,
\begin{align*}
&t=\frac{|x|\sin(u)}{\sqrt{(1-|x|\cos(u))^2+|x|^2\sin(u)}}\quad\Leftrightarrow\quad
t^2=\frac{|x|^2(1-\cos^2(u))}{1+|x|^2-2|x|\cos(u)}\\
&\Leftrightarrow\quad
|x|^2\cos^2(u)-2t^2|x|\cos(u)+t^2+t^2|x|^2-|x|^2=0\\
&\Leftrightarrow\quad \cos(u)=\frac{t^2\pm\sqrt{(1-t^2)(|x|^2-t^2)}}{|x|}.
\end{align*}
Above, the sign $\pm$ needs to be a plus sign. This is because
\begin{align*}
\frac{t^2-\sqrt{(1-t^2)(|x|^2-t^2)}}{|x|}\geq|x|
\quad\Leftrightarrow\quad
|x|^2-t^2\leq-\sqrt{(1-t^2)(|x|^2-t^2)},
\end{align*}
where the difference is $|x|^2-t^2$ non-negative for $t\leq|x|$. Thus, the root with a minus sign fulfills the condition $|x|<\cos(u)<1$ if and only if $t=|x|$ and, in this special case, it is equal to the root with a plus sign. In turn,
\begin{align*}
&|x|\leq\frac{t^2+\sqrt{(1-t^2)(|x|^2-t^2)}}{|x|}<1\\
&\Leftrightarrow\quad
|x|^2-t^2\leq\sqrt{(1-t^2)(|x|^2-t^2)}
\quad\text{and}\quad
\sqrt{(1-t^2)(|x|^2-t^2)}<|x|-t^2\\
&\Leftrightarrow\quad
|x|^2\leq1
\quad\text{and}\quad
0<(1-|x|)^2,
\end{align*}
which holds for all $x\in\B^2\setminus\{0\}$ and $0<t\leq|x|$. 

Consequently, we have
\begin{align*}
\cos(u)=\frac{t^2+\sqrt{(1-t^2)(|x|^2-t^2)}}{|x|},   
\end{align*}
and the points $y=|x|e^{(k\pm2u)i}$ can be written as $x(e^{ui})^{\pm1}$, where
\begin{align*}
e^{2ui}&=\cos(2u)+\sin(2u)i=2\cos^2(u)-1+2\sin(u)\cos(u)i\\
&=2\cos^2(u)-1+2\cos(u)\sqrt{1-\cos^2(u)}i.    
\end{align*}

Let us yet prove the rest of the lemma. For all such points $x,y\in\B^2$ that $|y|=|x|$, the inequality $s_{\B^2}(x,y)\leq|x|$ holds because, if $x,y$ are collinear with the origin, then by Lemma \ref{lem_sCollinear}
\begin{align*}
s_{\B^2}(x,y)=\frac{|x|+|y|}{2-||x|-|y||}=|x|,   
\end{align*}
and otherwise this inequality can be directly seen from Theorem \ref{thm_sConjugate}. Therefore, if $t>|x|$, $s_{\B^n}(x,y)<t$ for all $x\in\B^2\setminus\{0\}$ and $y\in S^1(0,|x|)$.
\end{proof}

Lemmas \ref{lem_colliIntS} and \ref{lem_conjugS} can be used to prove that certain inclusion results are sharp, as done also in the following theorem. 

\begin{theorem}\label{thm_sEucInc}
For $x\in G=\B^n$ and $0<t<1$,
$B^n(x,r_0)\subseteq B_s(x,t)\subseteq B^n(x,r_1)$ if and only if
\begin{align*}
0<r_0\leq\frac{2t(1-|x|)}{1+t}
\quad\text{and}\quad
r_1\geq\min\left\{\frac{2t(1-|x|)}{1-t},\frac{2t(1+|x|)}{1+t}\right\}
\end{align*}
\end{theorem}
\begin{proof}
For all $y\in\B^n(x,r_0)$ with $r_0$ as above, it follows from the fact that $B^n(x,1-|x|)\subseteq \B^n$ that
\begin{align*}
s_{\B^n}(x,y)\leq s_{B^n(x,1-|x|)}(x,y)=\frac{|x-y|}{2-2|x|-|x-y|}<\frac{r_0}{2-2|x|-r_0}\leq s.
\end{align*}
Consequently, for all $y\in\B^n(x,r_0)$, the triangular ratio distance between $x$ and $y$ is less than $t$ so $B^n(x,r_0)\subseteq B_s(x,t)$. Note that the upper limit of $r_0$ is also sharp. Namely, if the radius $r_0$ could be any greater than $2t(1-|x|)/(1+t)$, the ball $\B^n(x,r_0)$ would contain the point $y_0\in S_s(x,t)$ of Lemma \ref{lem_colliIntS} and this point is trivially outside of the ball $B_s(x,t)$.

For all $y\in\B^n\setminus\overline{B}^n(x,r_1)$ where $r_1$ is as in the theorem, it follows from the fact that $\B^n\subseteq B^n(x,1+|x|)\setminus\{x/|x|\}$ that
\begin{align*}
s_{\B^n}(x,y)
&\geq s_{B^n(x,1+|x|)\setminus\{x/|x|\}}(x,y)
\geq\max\{s_{\R^n\setminus\{x/|x|\}}(x,y),s_{B^n(x,1+|x|)}(x,y)\}\\
&=\max\left\{\frac{|x-y|}{1-|x|+|x/|x|-y|},\frac{|x-y|}{2+2|x|-|x-y|}\right\}\\
&\geq\max\left\{\frac{|x-y|}{2-2|x|+|x-y|},\frac{|x-y|}{2+2|x|-|x-y|}\right\}\\
&>\max\left\{\frac{r_1}{2-2|x|+r_1},\frac{r_1}{2+2|x|-r_1}\right\}\geq s,    
\end{align*}
so $B_s(x,t)\cap(\B^n\setminus\overline{B}^n(x,r_1))=\varnothing$. However, since $B_s(x,t)\subseteq\B^n$, this proves that $B_s(x,t)\subseteq B^n(x,r_1)$. It follows from this inclusion that $|x-y_1|\leq r_1$ for the point $y_1\in S_s(x,t)$ of Lemma \ref{lem_colliIntS}, which shows that the lower limit of $r_1$ in this theorem is sharp.
\end{proof}

Note that it is not possible find such a Euclidean ball $B^n(x,r)$ for the triangular ratio metric ball $B_s(x,t)$ that $B_s(x,t)\subseteq B^n(x,r)\subsetneq\B^n$ if the conditions of the following lemma do not hold.

\begin{lemma}\label{lem_eucInB}
For $x\in G=\B^n$ and $0<t<1$, the smallest Euclidean ball centered at $x$ and containing the triangular ratio metric ball $B_s(x,t)$ is included in the unit ball if and only if either $t<1/3$ and $t\leq|x|$, or $|x|<1/3$ and $|x|<t<(1-|x|)/(1+3|x|)$.
\end{lemma}
\begin{proof}
Let $B^n(x,r)$ be the smallest Euclidean ball centered at $x$ and containing $B_s(x,t)$. It is in the unit disk if and only if $|x|+r<1$. If $t\leq|x|$, $r=2t(1-|x|)/(1-t)$ here by Theorem \ref{thm_sEucInc}, and $|x|+2t(1-|x|)/(1-t)<1$ is equivalent to $t<1/3$. If $t>|x|$, $r=2t(1+|x|)/(1+t)$ by Theorem \ref{thm_sEucInc}, and $|x|+2t(1+|x|)/(1+t)<1$ if and only if $t<(1-|x|)/(1+3|x|)$. The conditions $t>|x|$ and $t<(1-|x|)/(1+3|x|)$ can both hold only if $|x|<(1-|x|)/(1+3|x|)$, which is equivalent to $3|x|^2+2|x|-1<0$. The equation $3|x|^2+2|x|-1=0$ has roots $|x|=-1$ and $|x|=1/3$ and, since $|x|$ is non-negative, $3|x|^2+2|x|-1<0$ holds with $|x|<1/3$.
\end{proof}

\begin{figure}[ht]
    \centering
    \begin{tikzpicture}
    \draw (0,0) circle (3cm);
    \draw (0,0) circle (0.08cm);
    \draw (0,1.5) circle (0.08cm);
    \draw (-1.763,2.427) circle (0.08cm);
    \draw (-1.191,-0.912) circle (0.08cm);
    \draw (0,0) -- (0,1.5);
    \draw (-3.091,3.125) -- (3.547,-0.365);
    \draw (-2.017,3.905) -- (-0.698,-3.782);
    \draw (-1.5*1.763,1.5*2.427) -- (1.5*1.763,-1.5*2.427);
    \draw (0,0.5) arc (90:127:0.5);
    \draw (-1.67,1.8) arc (275:310:0.5);
    \draw (-1.3,1.8) arc (310:335:0.8);
    \node[scale=1.3] at (0.3,0.1) {$0$};
    \node[scale=1.3] at (0,1.8) {$x$};
    \node[scale=1.3] at (-1.6,2.8) {$z$};
    \node[scale=1.3] at (-0.9,-0.9) {$y$};
    \node[scale=1.3] at (-0.2,0.7) {$u$};
    \node[scale=1.3] at (-1,1.75) {$v$};
    \node[scale=1.3] at (-1.4,1.5) {$v$};
    \end{tikzpicture}
    \caption{For $x=0.5i$, $t=0.5$, $u=\pi/5$, and $z=xe^{ui}/|x|$, a point $y$ found with Theorem \ref{thm_zus} so that $t=|x-y|/(|x-z|+|z-y|)$ and the line $L(0,z)$ bisects the angle between the lines $L(x,z)$ and $L(y,z)$.}
    \label{fig2}
\end{figure}
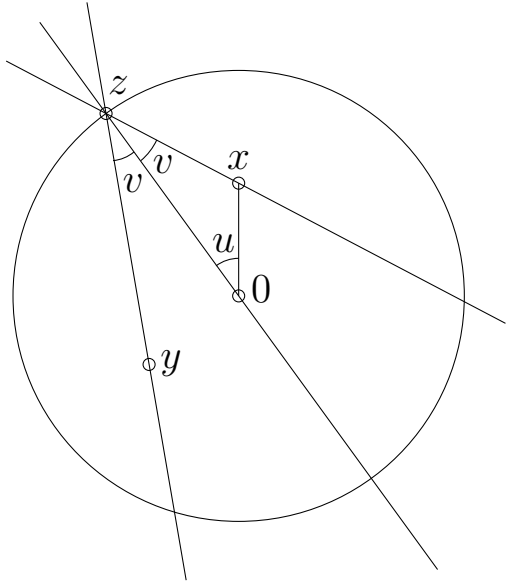

The following theorem is used to create an algorithm for drawing triangular ratio circles.

\begin{theorem}\label{thm_zus}
For $x\in\B^2\setminus\{0\}$, $0<t<1$, and $u\in[0,\pi]$, fix $z=xe^{ui}/|x|$. If
\begin{align*}
t<|x|\quad\text{and}\quad\frac{t^2-\sqrt{(1-t^2)(|x|^2-t^2)}}{|x|}<\cos(u)<\frac{t^2+\sqrt{(1-t^2)(|x|^2-t^2)}}{|x|},  
\end{align*}
there is no such point $y\in\B^2$ for which the equality
\begin{align*}
t=\frac{|x-y|}{|x-z|+|z-y|}    
\end{align*}
would hold and $L(0,z)$ would bisect the angle between $L(x,z)$ and $L(y,z)$. Otherwise, a point $y\in\B^2$ fulfills both of these conditions if and only if
\begin{align*}
y&=z(1-he^{-vi})\quad\text{with}\quad
v={\rm acos}\left(\frac{1-|x|\cos(u)}{|x-z|}\right)\in[0,\pi/2]\quad\text{and}\\
h&=\frac{|x-z|(t^2+\cos(2v)\pm2\cos(v)\sqrt{t^2-\sin^2(v)})}{1-t^2}\in(0,2\cos(v)).
\end{align*}
\end{theorem}
\begin{proof}
Fix $t$, $x$, $u$, and $z$ as in the lemma. Note that $u=\measuredangle XOZ$. Now,
\begin{align*}
|x-z|=||x|-e^{ui}|=\sqrt{1+|x|^2-2|x|\cos(u)}.    
\end{align*}
Let $v=\measuredangle XZO\in[0,\pi/2]$. We can solve that
\begin{align*}
\sqrt{1+|x-z|^2-2|x-z|\cos(v)}=|x|
\quad\Leftrightarrow\quad
\cos(v)=\frac{1+|x-z|^2-|x|^2}{2|x-z|}=\frac{1-|x|\cos(u)}{|x-z|}.
\end{align*}

For $y\in\B^2$, $L(0,z)$ bisects the angle between $L(x,z)$ and $L(y,z)$, if and only if $\measuredangle OZY=v$ but $y\notin L(x,z)$ unless $v=0$. See Figure \ref{fig2}. Consequently, we can write 
\begin{align*}
h&=z+h(-z)e^{-vi}=z(1-he^{-vi})
\end{align*}
for some $h\in\R$. The condition $y\in\B^2$ is equivalent to
\begin{align*}
&|y|=|1-he^{vi}|=\sqrt{1+h^2-2h\cos(v)}<1
\quad\Leftrightarrow\quad
h(h-2\cos(v))<0\\
&\Leftrightarrow\quad
0<h<2\cos(v).
\end{align*}

Suppose then that $t=|x-y|/(|x-z|+|z-y|)$. By using trigonometric identities $\cos(2v)=2\cos^2(v)-1$ and $\cos^2(v)+\sin^2(v)=1$, we will have
\begin{align*}
&t
=\frac{|x-y|}{|x-z|+|z-y|}
=\frac{\sqrt{h^2+|x-z|^2-2h|x-z|\cos(2v)}}{|x-z|+h}\\
&\Leftrightarrow\quad
(1-t^2)h^2-2|x-z|(t^2+\cos(2v))h+(1-t^2)|x-z|^2=0\\
&\Leftrightarrow\quad
h=\frac{2|x-z|(t^2+\cos(2v))\pm\sqrt{4|x-z|^2(t^2+\cos(2v))^2-4|x-z|^2(1-t^2)^2}}{2(1-t^2)}\\
&\quad\quad\quad\,=\frac{|x-z|(t^2+\cos(2v)\pm\sqrt{\cos^2(2v)+2t^2(\cos(2v)+1)-1})}{1-t^2}\\
&\quad\quad\quad\,=\frac{|x-z|(t^2+\cos(2v)\pm2\cos(v)\sqrt{t^2-\sin^2(v)})}{1-t^2}
\end{align*}
The roots above are defined if and only if 
\begin{align*}
&t^2\geq\sin^2(v)=1-\cos^2(v)=1-\frac{(1-|x|\cos(u))^2}{|x-z|^2}
=\frac{|x|^2(1-\cos^2(u))}{1+|x|^2-2|x|\cos(u)}\\
&\Leftrightarrow\quad
|x|^2\cos^2(u)-2t^2|x|\cos(u)+t^2+t^2|x|^2-|x|^2\geq0.
\end{align*}
If $t\geq|x|$, the inequality above holds for all $\cos(u)$. If $t<|x|$, we can solve that
\begin{align*}
&|x|^2\cos^2(u)-2t^2|x|\cos(u)+t^2+t^2|x|^2-|x|^2=0\\
&\Leftrightarrow\quad
\cos(u)=\frac{t^2+\sqrt{(1-t^2)(|x|^2-t^2)}}{|x|}
\end{align*}
and we need to choose
\begin{align*}
\cos(u)\leq\frac{t^2-\sqrt{(1-t^2)(|x|^2-t^2)}}{|x|}\quad\text{or}\quad\cos(u)\geq\frac{t^2+\sqrt{(1-t^2)(|x|^2-t^2)}}{|x|}    
\end{align*}
for $h$ to be well-defined. The proof is now finished.
\end{proof}

Note that even if $x,y\in\B^2$ and $z\in S^1$ so that $L(0,z)$ bisects the angle between the lines $L(x,z)$ and $L(y,z)$, this does not necessarily mean that 
\begin{align}\label{s_xyz}
s_{\B^2}(x,y)=\frac{|x-y|}{|x-z|+|z-y|}.    
\end{align}
However, if this equality holds for some $x,y\in\B^2$ and $z\in S^1$, then $L(0,z)$ must bisect the angle between the lines $L(x,z)$ and $L(y,z)$ by Theorem \ref{thm_findzinB}. In other words, when given $x\in\B^2$ and $u\in[0,\pi]$, Theorem \ref{thm_zus} can be only used to find potential candidates for such points $y\in\B^2$ that the equality \eqref{s_xyz} holds with the point $z=e^{ui}$. This is also the idea behind the following algorithm, which finds efficiently such points $y$ that belong to the circle $S_s(x,t)$. This algorithm has been used to create most of the images of this article, including Figure \ref{fig3} that shows how the radius $t$ affects the shape of the triangular ratio metric circles.

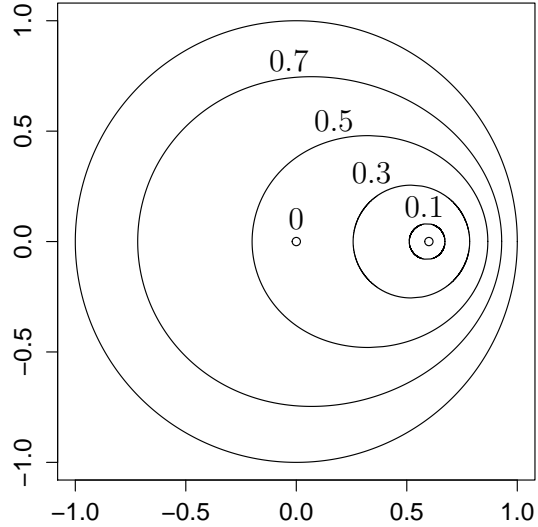
\begin{figure}[ht]
    \centering
    \begin{tikzpicture}
    \node[inner sep=0pt] at (0,0)
    {\includegraphics[width=7cm,trim={1cm 2.5cm 1cm 3cm},clip]{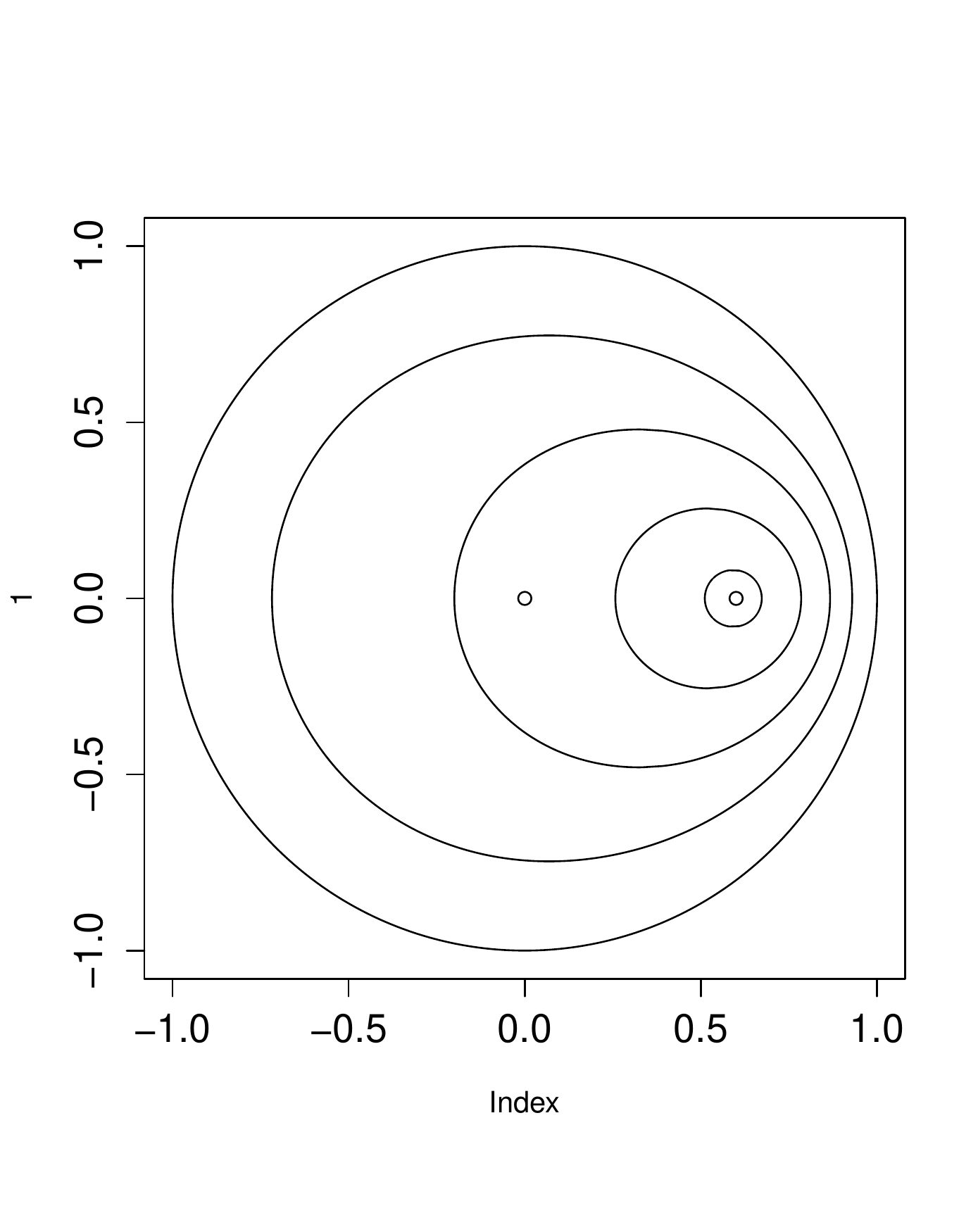}};
    \node[scale=1] at (0.3,0.6) {$0$};
    \node[scale=1] at (2,0.75) {$0.1$};
    \node[scale=1] at (1.3,1.2) {$0.3$};
    \node[scale=1] at (0.8,1.9) {$0.5$};
    \node[scale=1] at (0.2,2.7) {$0.7$};
    \end{tikzpicture}
    \caption{The triangular ratio metric circles $S_s(x,t)$ for the radii $t=0.1,0.3,0.5,0.7$, when the center $x$ is $0.6$ and the domain $G$ is the unit disk $\B^2$.}
    \label{fig3}
\end{figure}

\begin{nonsec}{\bf Algorithm for drawing triangular ratio circles in the unit disk}

\noindent The following instructions can be used to write an algorithm that plots the triangular ratio circle $S_s(x,t)$ in the domain $G=\B^2$ when the user gives the center $x\in\B^2$ and the radius $0<t<1$ of as an input.

1. If $x=0$, plot the Euclidean circle $S^1(x,2t/(1+t))$ and stop after that.

2. Otherwise, initialize a vector $Y$, fix for instance $N=10,000$ and $\epsilon=10^{-5}$, and define the functions
\begin{align*}
z(u)&=\frac{x}{|x|}e^{ui},\quad
v(u)={\rm acos}\left(\frac{1-|x|\cos(u)}{|x-z(u)|}\right)\\
h_0(u)&=\frac{|x-z(u)|(t^2+\cos(2v(u))-2\cos(v(u))\sqrt{t^2-\sin^2(v(u))})}{1-t^2},\quad\text{and}\\
h_1(u)&=\frac{|x-z(u)|(t^2+\cos(2v(u))+2\cos(v(u))\sqrt{t^2-\sin^2(v(u))})}{1-t^2}.
\end{align*}

3. If $t\geq|x|$, create a vector $U$ containing $N$ values from the interval $[0,\pi]$ and, if $t<|x|$ instead, choose the $N$ values of the vector $U$ from the set
\begin{align*}
\left[0,{\rm acos}\left(\frac{t^2+\sqrt{(1-t^2)(|x|^2-t^2)}}{|x|}\right)\right]\cup\left[{\rm acos}\left(\frac{t^2-\sqrt{(1-t^2)(|x|^2-t^2)}}{|x|}\right),\pi\right].    
\end{align*}

4. For each value $u\in U$, compute $h_0(u)$ and $h_1(u)$, add the point $z(u)(1-h_0(u)e^{-v(u)i})$ to the vector $Y$ if $0<h_0(u)<2\cos(v(u))$ and, similarly, add the point $z(u)(1-h_1(u)e^{-v(u)i})$ to the vector $Y$ if $0<h_1(u)<2\cos(v(u))$.

5. Then for each value $y\in Y$, compute the triangular ratio distance $s_{\B^2}(x,y)$ by using for instance the algorithm \cite[Algorithm 2.5, p. 686]{chkv} and remove this value $y$ from $Y$ if $|s_{\B^2}(x,y)-t|>\epsilon$.

6. For each remaining element $y\in Y$ to the vector $Y$, add the value $x\overline{y}/\overline{x}$ to the vector $Y$ because this is a reflection of $y$ over the line $L(0,x)$ by \cite[B.11, p. 460]{hkv} and the triangular ratio metric is invariant under reflection.

7. Organize all the values $y\in Y$ so that the complex argument of the difference $y-x$ increases monotonically on $[0,2\pi)$ as the one lists the elements of the vector $Y$ and plot the triangular ratio circle $S_s(x,t)$ by connecting the organized points of $Y$. 
\end{nonsec}

For $x\in G=\B^3$ and $0<t<1$, the three-dimensional triangular ratio sphere $S_s(x,t)$ can be drawn by using the algorithm above to plot the two-dimensional circle $S_s(x,t)$ in any two-dimensional plane containing both $x$ and the origin, and then plot the three-dimensional sphere by rotating this circle around either the line $L(0,x)$ or any line that contains the origin if $x=0$. 

\section{Triangular ratio metric and $j^*$-metric balls}

Unlike for triangular ratio metric balls, an explicit expression can be found for any $j^*$-metric sphere $S_{j^*}(x,k)$ in terms of Euclidean geometry in the domain $G=\B^n$, and the results of Theorem \ref{thm_yCirc} and Corollary \ref{cor_jBEucS} can be used to create an algorithm drawing $j^*$-metric balls in the unit disk or ball.

\begin{theorem}\label{thm_yCirc}
For $x\in\B^n\setminus\{0\}$ and $0<k<1$, the sphere $S_\Upsilon(x,k)$ defined with the function
\begin{align*}
\Upsilon_{\B^n}(x,y)=\frac{|x-y|}{|x-y|+2-2|y|}
\end{align*}
in the domain $\B^n\setminus\{0\}$ is
\begin{align*}
&\left\{y\in\B^n\setminus\{0\}\text{ }|\text{ }|y|=\frac{1-|x|^2}{2(1-|x|\cos(u))}\quad\text{for}\quad u=\measuredangle XOY\right\}
\quad\text{if}\quad k=\frac{1}{3},\\
&\left\{y\in\B^n\setminus\{0\}\text{ }|\text{ }|y|=l_0(u)\quad\text{for}\quad u=\measuredangle XOY\right\}
\quad\text{if}\quad k<\frac{1}{3}\quad\text{and}\quad |x|<\frac{2k}{1-k},\\
&\left\{y\in\B^n\setminus\{0\}\text{ }|\text{ }|y|\in\{l_0(u),l_1(u)\},\,\cos(u)\geq c_1\quad\text{for}\quad u=\measuredangle XOY\right\}
\quad\text{if}\quad k<\frac{1}{3}\\ &\text{and}\quad|x|>\frac{2k}{1-k},\quad\text{or}\\
&\left\{y\in\B^n\setminus\{0\}\text{ }|\text{ }|y|=l_0(u),\,\cos(u)\leq c_0\quad\text{for}\quad u=\measuredangle XOY\right\}
\quad\text{if}\quad k>\frac{1}{3},
\end{align*}
where
\begin{align*}
c_0&=\frac{4k^2-\sqrt{(3k^2+2k-1)(4k^2-|x|^2(1-k)^2)}}{|x|(1-k)^2},\\
c_1&=\frac{4k^2+\sqrt{(3k^2+2k-1)(4k^2-|x|^2(1-k)^2)}}{|x|(1-k)^2},\\
l_0(u)&=\frac{4k^2-(1-k)^2|x|\cos(u)-(1-k)\sqrt{4k^2(1+|x|^2-2|x|\cos(u))-(1-k)^2|x|^2\sin^2(u)}}{3k^2+2k-1},\\
l_1(u)&=\frac{4k^2-(1-k)^2|x|\cos(u)+(1-k)\sqrt{4k^2(1+|x|^2-2|x|\cos(u))-(1-k)^2|x|^2\sin^2(u)}}{3k^2+2k-1}.
\end{align*}
\end{theorem}
\begin{proof}
For $y\in\B^n\setminus\{0\}$, let $u=\measuredangle XOY$. By the cosine formula,
\begin{align*}
|x-y|=\sqrt{|x|^2+|y|^2-2|x||y|\cos(u)}.
\end{align*}
Consequently,
\begin{align*}
&\Upsilon_{\B^n}(x,y)=\frac{\sqrt{|x|^2+|y|^2-2|x||y|\cos(u)}}{\sqrt{|x|^2+|y|^2-2|x||y|\cos(u)}+2(1-|y|)}=k\\ 
&\Leftrightarrow\quad
(1-k)\sqrt{|x|^2+|y|^2-2|x||y|\cos(u)}=2k(1-|y|)\\
&\Leftrightarrow\quad
(3k^2+2k-1)|y|^2-2(4k^2-(1-k)2|x|\cos(u))|y|+4k^2-(1-k)^2=0.
\end{align*}
Note that, for $k>0$, $3k^2+2k-1>0$ if and only if $k>1/3$. If $k=1/3$, we have $|y|=(1-|x|^2)/(2(1-|x|\cos(u)))$. If $k\neq1/3$,
\begin{align}\label{rootForRInJ}
|y|=\frac{4k^2-(1-k)^2|x|\cos(u)\pm(1-k)\sqrt{4k^2(1+|x|^2-2|x|\cos(u))-(1-k)^2|x|^2\sin^2(u)}}{3k^2+2k-1}.    
\end{align}
The root above is defined if and only if
\begin{align*}
&4k^2(1+|x|^2-2|x|\cos(u))-(1-k)^2|x|^2\sin^2(u)\geq0\\  
&4k^2(1+|x|^2-2|x|\cos(u))\geq(1-k)^2|x|^2(1-\cos^2(u))\\
&(1-k)^2|x|^2\cos^2(u)-8k^2|x|\cos(u)+4k^2(1+|x|^2)-(1-k)^2|x|^2\geq0.
\end{align*}
The discriminant in the inequality above is
\begin{align*}
&(-8k^2|x|\cos(u))^2-4(1-k)^2|x|^2(4k^2(1+|x|^2)-(1-k)^2|x|^2)\\
&=4|x|^2(3k^2+2k-1)(4k^2-|x|^2(1-k)^2).
\end{align*}
Note that if $k>1/3$, then $(4k^2-|x|^2(1-k)^2)>4(1-|x|^2)/9>0$. It follows that the root \eqref{rootForRInJ} is defined if either
\begin{equation}\label{condForL}
\begin{aligned}
&k<\frac{1}{3}\quad\text{and}\quad|x|<\frac{2k}{1-k},\quad\text{or}\quad\\
&\cos(u)\leq\frac{4k^2-\sqrt{(3k^2+2k-1)(4k^2-|x|^2(1-k)^2)}}{|x|(1-k)^2}\quad\text{or}\\
&\cos(u)\geq\frac{4k^2+\sqrt{(3k^2+2k-1)(4k^2-|x|^2(1-k)^2)}}{|x|(1-k)^2}.
\end{aligned}
\end{equation}

Suppose that the condition \eqref{condForL} holds and the root \eqref{rootForRInJ} is defined. Let us first consider the case where the sign $\pm$ of this root is a plus. Clearly, $y\in\B^n\setminus\{0\}$ if and only if $0<|y|<1$. If $k>1/3$,
\begin{align*}
&\frac{4k^2-(1-k)^2|x|\cos(u)+(1-k)\sqrt{4k^2(1+|x|^2-2|x|\cos(u))-(1-k)^2|x|^2\sin^2(u)}}{3k^2+2k-1}<1\\
&\Leftrightarrow\quad
(1-k)\sqrt{4k^2(1+|x|^2-2|x|\cos(u))-(1-k)^2|x|^2\sin^2(u)}<-(1-k)^2(1-|x|\cos(u)),   
\end{align*}
which does not hold. If $k<1/3$ instead,
\begin{align*}
&\frac{4k^2-(1-k)^2|x|\cos(u)+(1-k)\sqrt{4k^2(1+|x|^2-2|x|\cos(u))-(1-k)^2|x|^2\sin^2(u)}}{3k^2+2k-1}>0\\
&\Leftrightarrow\quad
(1-k)\sqrt{4k^2(1+|x|^2-2|x|\cos(u))-(1-k)^2|x|^2\sin^2(u)}<-4k^2+(1-k)^2|x|\cos(u),
\end{align*}
which can only if $\cos(u)>4k^2/(|x|(1-k)^2)$. If this inequality holds and $k<1/3$,
\begin{align*}
&(1-k)\sqrt{4k^2(1+|x|^2-2|x|\cos(u))-(1-k)^2|x|^2\sin^2(u)}<-4k^2+(1-k)^2|x|\cos(u)\\
&\Leftrightarrow\quad-(3k^2+2k-1)(4k^2-|x|^2(1-k)^2)<0\\
&\Leftrightarrow\quad4k^2-|x|^2(1-k)^2<0
\quad\Leftrightarrow\quad|x|>2k/(1-k).
\end{align*}
By combining the inequalities $k<1/3$, $|x|>2k/(1-k)$, and $\cos(u)>4k^2/(|x|(1-k)^2)$ to the condition \eqref{condForL}, it follows that we can choose a plus for the sign $\pm$ of the root \eqref{rootForRInJ} only if
\begin{align*}
k<\frac{1}{3},\quad
|x|>\frac{2k}{1-k},\quad\text{and}\quad
\cos(u)\geq\frac{4k^2+\sqrt{(3k^2+2k-1)(4k^2-|x|^2(1-k)^2)}}{(1-k)^2|x|}.
\end{align*}

Next, we will study the root \eqref{rootForRInJ} with minus. Suppose again that the condition \eqref{condForL} hold. It can be verified that
\begin{align*}
\frac{4k^2-(1-k)^2|x|\cos(u)-(1-k)\sqrt{4k^2(1+|x|^2-2|x|\cos(u))-(1-k)^2|x|^2\sin^2(u)}}{3k^2+2k-1}<1  
\end{align*}
regardless of whether $k<1/3$ or $k>1/3$. If $k>1/3$,
\begin{align*}
&\frac{4k^2-(1-k)^2|x|\cos(u)-(1-k)\sqrt{4k^2(1+|x|^2-2|x|\cos(u))-(1-k)^2|x|^2\sin^2(u)}}{3k^2+2k-1}>0\\
&\Leftrightarrow\quad
4k^2-(1-k)^2|x|\cos(u)>(1-k)\sqrt{4k^2(1+|x|^2-2|x|\cos(u))-(1-k)^2|x|^2\sin^2(u)}.
\end{align*}
This cannot hold if $\cos(u)\geq4k^2/(|x|(1-k)^2)$. If $k>1/3$ and $\cos(u)<4k^2/(|x|(1-k)^2)$, then
\begin{align*}
&4k^2-(1-k)^2|x|\cos(u)>(1-k)\sqrt{4k^2(1+|x|^2-2|x|\cos(u))-(1-k)^2|x|^2\sin^2(u)}\\
&\Leftrightarrow\quad
(3k^2+2k-1)(4k^2-|x|^2(1-k)^2)>0
\quad\Leftrightarrow\quad
|x|>2k/(1-k),
\end{align*}
which is true since $2k/(1-k)>1$ for all $k>1/3$. Similarly, if $k<1/3$,
\begin{align*}
&\frac{4k^2-(1-k)^2|x|\cos(u)-(1-k)\sqrt{4k^2(1+|x|^2-2|x|\cos(u))-(1-k)^2|x|^2\sin^2(u)}}{3k^2+2k-1}>0\\
&\Leftrightarrow\quad
4k^2-(1-k)^2|x|\cos(u)<(1-k)\sqrt{4k^2(1+|x|^2-2|x|\cos(u))-(1-k)^2|x|^2\sin^2(u)},
\end{align*}
which holds if either $\cos(u)\geq4k^2/(|x|(1-k)^2)$ or $|x|<2k/(1-k)$. By combining these observations with the condition \eqref{condForL}, we have
\begin{align*}
&k<\frac{1}{3}\quad\text{and}\quad|x|<\frac{2k}{1-k},
\quad\text{or}\\
&k<\frac{1}{3},\quad|x|>\frac{2k}{1-k},\quad\text{and}\quad
\cos(u)\geq\frac{4k^2+\sqrt{(3k^2+2k-1)(4k^2-|x|^2(1-k)^2)}}{(1-k)^2|x|},\quad\text{or}\\
&k>\frac{1}{3},\quad\text{and}\quad
\cos(u)\leq\frac{4k^2-\sqrt{(3k^2+2k-1)(4k^2-|x|^2(1-k)^2)}}{(1-k)^2|x|}.
\end{align*}
The theorem follows now.
\end{proof}

The result of the following corollary can be better understood by looking at Figure \ref{fig4}.

\begin{figure}[ht]
    \centering
    \begin{tikzpicture}
    \node[inner sep=0pt] at (0,0)
    {\includegraphics[width=7cm,trim={1cm 2.5cm 1cm 3cm},clip]{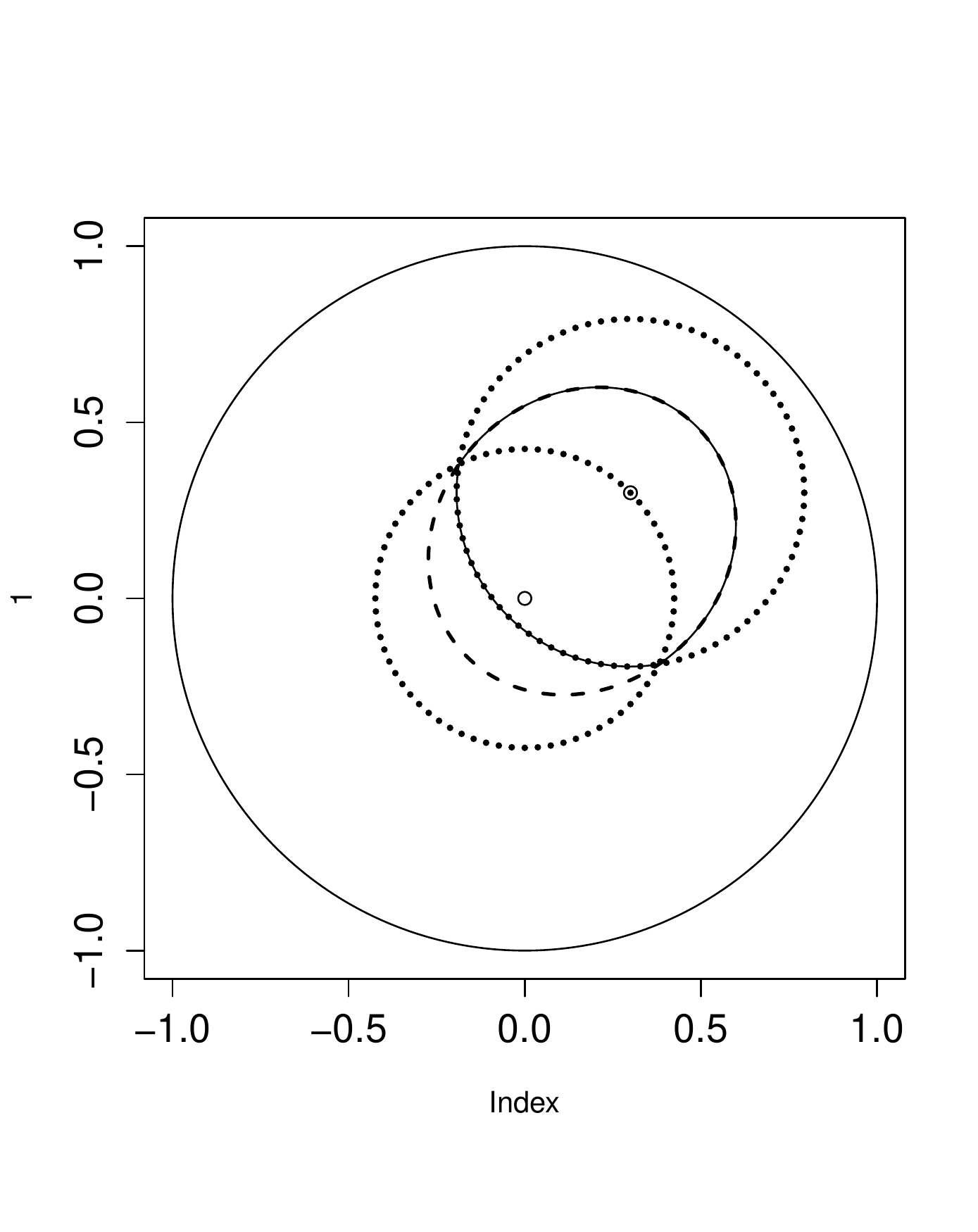}};
    \node[scale=1.3] at (1.2,1.5) {$x$};
    \node[scale=1.3] at (0.3,0.6) {$0$};
    \end{tikzpicture}
    \caption{For $x=0.3+0.3i$ and $k=0.3$, the Euclidean circles $S^1(0,|x|)$ and $S^1(x,2k(1-|x|)/(1-k))$ as dotted line, the circle $S_\Upsilon(x,k)$ defined with the function $\Upsilon_{\B^2}(x,y)$ of Theorem \ref{thm_yCirc} as dashed line, and the $j^*$-metric circle $S_{j^*}(x,k)$ as solid line in the domain $G=\B^2$.}
    \label{fig4}
\end{figure}

\begin{corollary}\label{cor_jBEucS}
For $x\in G=\B^n$ and $0<k<1$,
\begin{align*}
&S_{j^*}(x,k)=S^{n-1}(0,2k/(1+k))\quad\text{if}\quad x=0,\quad\text{and otherwise}\\
&S_{j^*}(x,k)=
\left(S^{n-1}\left(x,\frac{2k(1-|x|)}{1-k}\right)\cap\overline{B}^n(0,|x|)\right)\cup
(S_\Upsilon(x,k)\setminus\overline{B}^n(0,|x|)),\\
\end{align*}
where $S_\Upsilon(x,k)$ is as in Theorem \ref{thm_yCirc}.
\end{corollary}
\begin{proof}
The first part of the corollary follows directly from Lemma \ref{lem_sForx0} because $j^*_{\B^n}(y,0)=s_{\B^n}(y,0)$ for all $y\in\B^n$.

Suppose below that $x\neq0$ and denote $r=|x-y|$. Consider first the part of the $j^*$-metric sphere $S_{j^*}(x,k)$ in the closure of the Euclidean metric ball $B^n(0,|x|)$. For all $y\in\overline{B}^n(0,|x|)$, $|y|\leq|x|$ and $j^*_{\B^n}(x,y)=r/(r+2-2|x|)=k$ holds if and only if $r=2k(1-|x|)/(1-k)$. Consequently,
\begin{align*}
S_{j^*}(x,k)\cap\overline{B}^n(0,|x|)=S^{n-1}\left(x,\frac{2k(1-|x|)}{1-k}\right)\cap\overline{B}^n(0,|x|).   
\end{align*}

If $|y|>|x|$, then the distance $j^*(x,y)$ is equivalent to the value $\Upsilon_{\B^n}(x,y)$ of the function introduced in Theorem \ref{thm_yCirc}, so
\begin{align*}
S_{j^*}(x,k)\setminus\overline{B}^n(0,|x|)=S_\Upsilon(x,k)\setminus\overline{B}^n(0,|x|),  
\end{align*}
and the corollary follows.
\end{proof}

\begin{remark}
Note that Corollary \ref{cor_jBEucS} can be used to find the balls in the distance ratio metric, too, given $j_{\B^n}(x,y)={\rm th}(j^*_{\B^n}(x,y)/2)$.
\end{remark}

\begin{lemma}\label{lem_jBInc}
For $x\in G=\B^n$ and $0<r<1-|x|$, $B_{j^*}(x,k_0)\subseteq B^n(x,r)\subseteq B_{j^*}(x,k_1)$ if and only if
\begin{align*}
0<k_0\leq\frac{r}{2-|r-2|x||}
\quad\text{and}\quad
k_1\geq\frac{r}{2-2|x|-r}.
\end{align*}
\end{lemma}
\begin{proof}
The $j^*$-metric ball $B_{j^*}(x,k_0)$ is included in the Euclidean ball $B^n(x,r)$ if and only if $k_0\leq j^*_{\B^n}(x,y)$ for all $y\in S^{n-1}(x,r)$, and $B_{j^*}(x,k_1)$ contains $B^n(x,r)$ if and only if $j^*_{\B^n}(x,y)\leq k_1$ for $y\in S^{n-1}(x,r)$. For all $y\in S^{n-1}(x,r)$,
\begin{align*}
j^*_{\B^n}(x,y)=\frac{r}{r+2-2\max\{|x|,\sqrt{r^2+|x|^2-2r|x|\cos(u)}\}},    
\end{align*}
where $u=\measuredangle OXY$ or, if $x=0$, $u=\measuredangle ZOY$ with $z=1$. Clearly, the minimum value of the distance $j^*_{\B^n}(x,y)$ with respect to $y\in S^{n-1}(x,r)$ is
\begin{align*}
\frac{r}{r+2-2\max\{|x|,||x|-r|\}} 
=\begin{cases}
r/(r+2-2|x|)\text{ if }r\leq2|x|,\\
r/(r+2-2(r-|x|))\text{ if }r>2|x|
\end{cases}
\end{align*}
and the maximum value $r/(2-2|x|-r)$, from which the lemma follows. 
\end{proof}

\begin{corollary}
For $x\in G=\B^n$ and $0<k<1$, $B^n(x,r_0)\subseteq B_{j^*}(x,k)\subseteq B^n(x,r_1)$ if and only if
\begin{align*}
0<r_0\leq\frac{2k(1-|x|)}{1+k}\quad\text{and}\quad
r_1\geq\min\left\{\frac{2k(1-|x|)}{1-k},\frac{2k(1+|x|)}{1+k}\right\}
\end{align*}
\end{corollary}
\begin{proof}
Follows from Lemma \ref{lem_jBInc}.
\end{proof}

The following inequality helps to formulate a result about the ball inclusion between the triangular ratio metric and the $j^*$-metric.

\begin{corollary}\label{cor_sjIne}
For any domain $G\subsetneq\R^n$ and all $x,y\in G$, 
\begin{align*}
s_G(x,y)\geq\frac{j^*_G(x,y)}{1-j^*_G(x,y)}.    
\end{align*}
\end{corollary}
\begin{proof}
By \cite[Lemma 2.1, p. 1124]{hvz}
\begin{align*}
s_G(x,y)\leq\frac{e^{j_G(x,y)}-1}{2}=\frac{e^{2{\rm arth}(j^*_G(x,y))}-1}{2}=\frac{j^*_G(x,y)}{1-j^*_G(x,y)}.    
\end{align*}
\end{proof}

\begin{lemma}\label{lem_jSBinG}
For any domain $G\subsetneq\R^n$, $x\in G$, and $0<t<1$,
\begin{align*}
B_{j^*}\left(x,t/(1+t)\right)\subseteq B_s(x,t)\subseteq B_{j^*}(x,t), 
\end{align*}
where the latter inclusion above is sharp for all possible choices of a domain $G$, and the former inclusion can be replaced with
\begin{align*}
B_{j^*}\left(x,t/\min\{(1+t),\sqrt{2}\}\right)\subseteq B_s(x,t) 
\end{align*}
if $G$ is convex.
\end{lemma}
\begin{proof}
The inequality of Corollary \ref{cor_sjIne} is equivalent to $j^*_G(x,y)\leq s_G(x,y)/(1+s_G(x,y))$ and, by Theorem \ref{thm_sjbounds}, the inequality $j^*_G(x,y)\leq s_G(x,y)$ holds for all $x,y\in G$. It now follows that
\begin{align*}
&B_{j^*}(x,t/(1+t))=\{y\in\B^n|j^*_{\B^n}(x,y)<t/(1+t)\} \subseteq\{y\in\B^n|s_{\B^n}(x,y)<t\}=B_s(x,t)\\ 
&\subseteq B_{j^*}(x,t).    
\end{align*}
To show the inclusion $B_s(x,t)\subseteq B_{j^*}(x,t)$ is sharp, fix $x\in G$, $z\in S^{n-1}(x,d_G(x))\cap\partial G$, and $y=x+2t(z-x)/(1+t)$ and note that
\begin{align*}
s_G(x,y)
=\frac{|x-y|}{|x-z|+|z-y|}
=t
=\frac{|x-y|}{|x-y|+2\min\{d_G(x),d_G(y)\}}
=j^*_G(x,y).
\end{align*}
The last part of the lemma follows from Theorem \ref{thm_sjbounds}.
\end{proof}

Given $j^*_G(x,y)={\rm th}(j_G(x,y)/2)$, the $j^*$-metric in the former result can be easily replaced with the distance ratio metric.

\begin{corollary}\label{cor_sjG}
For any domain $G\subsetneq\R^n$, $x\in G$, and $0<t<1$,
\begin{align*}
B_j(x,\log(1+2t))\subseteq B_s(x,t)\subseteq B_j(x,\log((1+t)/(1-t))),
\end{align*}
where latter inclusion above is sharp for all possible choices of a domain $G$ and the former inclusion can be replaced with
\begin{align*}
B_j\left(x,\max\left\{\log\left(\frac{t+\sqrt{2}}{\sqrt{2}-t}\right),\log(1+2t)\right\}\right)\subseteq B_s(x,t)    
\end{align*}
if $G$ is convex.
\end{corollary}
\begin{proof}
Follows from Lemma \ref{lem_jSBinG} and the formula \eqref{def_j}.
\end{proof}

\begin{remark}
Corollary \ref{cor_sjG} proves an earlier conjecture in \cite[Conj. 7.7, p. 122]{h16}, and also improves the constant in the first of this conjecture.
\end{remark}

As noted already in Lemma \ref{lem_eucInB}, there is not always a Euclidean ball that is included in the unit ball, contains another ball defined with some hyperbolic type metric, and has the same center as the other ball, the following lemma offers a non-$x$-centered Euclidean ball that contains the balls $B_s(x,t)$ and $B_{j^*}(x,t)$ and is included in the unit ball if $t\geq|x|$.

\begin{lemma}\label{lem_bsjEuc}
For $G=\B^n$, $x\in\B^n\setminus\{0\}$, and $|x|\leq t<1$, $B_s(x,t)\subseteq B_{j^*}(x,t)\subseteq B^n(q,r)\subsetneq\B^n$ if
\begin{align*}
q=x-\frac{2tx}{1+t}
\quad\text{and}\quad
r=\frac{2t}{1+t}
\end{align*}
and $B^n(q,r)$ here is the smallest Euclidean ball containing the ball $B_s(x,t)$.
\end{lemma}
\begin{proof}
By Lemma \ref{lem_jSBinG}, $B_s(x,t)\subseteq B_{j^*}(x,t)$. Let $y_0$ and $y_1$ be the intersection points in $S_s(x,t)\cap L(0,x)$ found in Lemma \ref{lem_colliIntS}. Since $t\geq|x|$, 
\begin{align*}
y_0=x+\frac{2tx(1-|x|)}{|x|(1+t)}\quad\text{and}\quad
y_1=x-\frac{2tx(1+|x|)}{|x|(1+t)}    
\end{align*}
and it can be easily verified that $q=(y_0-y_1)/2$ and $r=|q-y_0|/2$. Consequently, $B^n(q,r)\subsetneq\B^n$ because $|q|+r=|y_0|<1$. Let us now prove that $B_{j^*}(x,t)\subseteq B^n(q,r)$ by showing that, for all $y\in S^{n-1}(q,r)$,
\begin{align*}
j^*_{\B^n}(x,y)=\left(1+\frac{2(1-\max\{|x|,|y|\})}{|x-y|}\right)^{-1}>t
\quad\Leftrightarrow\quad
\frac{1-\max\{|x|,|y|\}}{|x-y|}<\frac{1-t}{2t}
\end{align*}
If $u=\measuredangle XQY$, $\measuredangle OQY=\pi-u$ since $q$ is between the origin and the point $x$. By cosine formula,
\begin{align*}
|y|&=\sqrt{|q|^2+r^2+2r|q|\cos(u)}=\frac{1}{1+t}\sqrt{|x|(1-t)^2+4t^2+4t(1-t)|x|\cos(u)},\\
|x-y|&=\sqrt{|x-q|^2+r^2-2r|x-q|\cos(u)}=\frac{2t}{1+t}\sqrt{1+|x|^2-2|x|\cos(u)}.
\end{align*}
Clearly, the distances $|x|$, $|y|$, and $|x-y|$ are all increasing with respect to $|x|$ and therefore 
\begin{align*}
&\frac{1-\max\{|x|,|y|\}}{|x-y|}
=\frac{1+t-\max\{(1+t)|x|,\sqrt{|x|(1-t)^2+4t^2+4t(1-t)|x|\cos(u)}\}}{2t\sqrt{1+|x|^2-2|x|\cos(u)}}\\
&<\frac{1+t-\max\{(1+t)\cdot0,\sqrt{0\cdot(1-t)^2+4t^2+4t(1-t)\cdot0\cdot\cos(u)}\}}{2t\sqrt{1+0^2-2\cdot0\cdot\cos(u)}}
=\frac{1-t}{2t}.
\end{align*}
Thus, the inclusion of the lemma holds and, since $y_0,y_1\in S_s(x,t)$ and $|y_0-y_1|=2r$, there is no such Euclidean ball with smaller radius than $r$ that contains $B_s(x,t)$.
\end{proof}

\begin{remark}
If $G=\B^n$, $x\in\B^n\setminus\{0\}$, and $0<t<|x|$ and $q=(y_0-y_1)/2$ and $r=|q-y_0|/2$ for the points $y_0,y_1$ of Lemma \ref{lem_colliIntS}, the inclusion $B_s(x,t)\subseteq B^n(q,r)\subsetneq\B^n$ holds but $B_{j^*}(x,t)\nsubseteq B^n(q,r)$ according to computer tests.
\end{remark}

\section{Triangular ratio metric balls in hyperbolic geometry}

Since the hyperbolic metric is conformally invariant unlike the $j^*$-metric or the triangular ratio metric, it is useful to have some ball inclusion results between the hyperbolic metric and these two metrics.

\begin{corollary}\label{cor_rhoIncj}
For $x\in G=\B^n$ and $R>0$, $B_{j^*}(x,k_0)\subseteq B_\rho(x,R)\subseteq B_{j^*}(x,k_1)$ if and only if
\begin{align*}
&0<k_0\leq\max\left\{\frac{(1+|x|){\rm sh}(R/2)}{2+(1+|x|){\rm sh}(R/2)},\frac{(1-|x|)(e^R-1)}{3+e^R-|x|(e^R-1)}\right\}
\quad\text{and}\\
&k_1\geq\frac{(1+|x|)(e^R-1)}{3+e^R+|x|(e^R-1)},
\end{align*}
and, for $x\in G=\B^n$ and $0<k<1$, $B_\rho(x,R_0)\subseteq B_{j^*}(x,k)\subseteq B_\rho(x,R_1)$ if and only if
\begin{align*}
&0<R_0\leq\log\left(1+\frac{4k}{(1-k)(1+|x|)}\right)
\quad\text{and}\\
&R_1\geq\min\left\{2{\rm arsh}\left(\frac{2k}{(1-k)(1+|x|)}\right),\log\left(1+\frac{4k}{(1-k)(1-|x|)}\right)\right\}. 
\end{align*}
\end{corollary}
\begin{proof}
Follows from Theorem \ref{thm_distRhoInc} and the formula \eqref{def_j}.
\end{proof}

The next result is very trivial but needed to check that other results hold in the special case $x=0$.

\begin{lemma}\label{lem_sBrhoOrig}
For $G=\B^n$ and $0<t<1$, $B_s(0,s)=B_\rho(0,\log((1+3t)/(1-t)))$.
\end{lemma}
\begin{proof}
By formula \eqref{def_rho} and Lemma \ref{lem_sForx0},
\begin{align*}
\rho_{\B^n}(y,0) 
=2{\rm arsh}\left(\frac{|y|}{\sqrt{1-|y|^2}}\right)
=2{\rm arsh}\left(\frac{2t}{\sqrt{(1-t)(1+3t)}}\right)
=\log\left(\frac{1+3t}{1-t}\right).
\end{align*}
\end{proof}

Recall that, if $G=\B^n$, we have explicit formulas for two points $y_0,y_1\in S_s(x,t)$ with any $x\in\B^n\setminus\{0\}$ and $0<t<1$ by Lemma \ref{lem_colliIntS}, from which follows that we can create upper and lower bounds for the radii in the inclusion of the triangular ratio metric balls and hyperbolic balls.

\begin{lemma}\label{lem_sRhoBounds}
For $x\in G=\B^n$ and $R>0$, $B_s(x,t_0)\subseteq B_\rho(x,R)\subseteq B_s(x,t_1)$ can hold only if
\begin{align*}
&0<t_0\leq\frac{{\rm th}(R/2)(1-|x|^2)}{2(1-|x|{\rm th}(R/2))-|2|x|-{\rm th}(R/2)(1+|x|^2)|}
\quad\text{and}\\
&t_1\geq\frac{(1+|x|)(e^R-1)}{3+e^R+|x|(e^R-1)},
\end{align*}
and, for $x\in G=\B^n$ and $0<t<1$, $B_\rho(x,R_0)\subseteq B_s(x,t)\subseteq B_\rho(x,R_1)$ can hold only if
\begin{align*}
&0<R_0\leq\log\left(1+\frac{4t}{(1-t)(1+|x|)}\right)
\quad\text{and}\quad
R_1\geq2{\rm arsh}\left(\frac{2t}{\sqrt{l(t,|x|)}}\right),
\quad\text{where}\\
&l(t,|x|)=\max\{(1+t)(1+|x|)(1-3t+|x|(1+t)),\\
&\quad\quad\quad\quad\quad\quad\quad
(1-t)(1-|x|)(1+3t-|x|(1-t))\}.
\end{align*}
\end{lemma}
\begin{proof}
Let us first consider the latter part of the result. Suppose first that $x\in\B^n\setminus\{0\}$. If $y_0,y_1\in S_s(x,t)$ are as in Lemma \ref{lem_colliIntS},
\begin{align*}
&\rho_{\B^n}(x,y_0)
=\log\left(1+\frac{4t}{(1-t)(1+|x|)}\right),\\
&\rho_{\B^n}(x,y_1)
=2{\rm arsh}\left(\frac{2t}{\sqrt{(1+t)(1+|x|)(1-3t+|x|(1+t))}}\right)\quad\text{if}\quad t<|x|,\\
&\rho_{\B^n}(x,y_1)
=\log\left(1+\frac{4t}{(1-t)(1-|x|)}\right)\quad\text{if}\quad t\geq|x|.
\end{align*}
Clearly, $B_\rho(x,R_0)\subseteq B_s(x,t)\subseteq B_\rho(x,R_1)$ can only hold if $R_0\leq\rho_{\B^n}(x,y_0)$ and $R_1\geq\rho_{\B^n}(x,y_1)$.
Note that, since
\begin{align*}
\left(1+\frac{4t}{(1-t)(1-|x|)}\right)
=2{\rm arsh}\left(\frac{2t}{\sqrt{(1-t)(1-|x|)(1+3t-|x|(1-t))}}\right),
\end{align*}
the distance $\rho_{\B^n}(x,y_1)$ can be written as $2{\rm arsh}(2t/\sqrt{l(t,|x|})$ where $l(t,|x|)$ is as in the lemma. Because these bounds of $R_0$ and $R_1$ hold also in the special case $x=0$ by Lemma \ref{lem_sBrhoOrig}, the latter part of the result follows.

Now, we can solve that
\begin{align*}
R\leq\log\left(1+\frac{4t}{(1-t)(1+|x|)}\right) 
\quad\Leftrightarrow\quad
t\geq\frac{(1+|x|)(e^R-1)}{3+e^R+|x|(e^R-1)},
\end{align*}
which proves the lower bound of $t_1$.

For $x\in\B^n\setminus\{0\}$, suppose that $R=\rho_{\B^n}(x,y_1)$ with $y_1$ as in Lemma \ref{lem_colliIntS} for $0<t<1$. Now, the circles $S_\rho(x,R)$ and $S_s(x,t)$ intersect at the point $y_1$. Since $y_1=S_\rho(x,R)\cap L(0,x)\cap B(0,|x|)$, it follows from Lemma \ref{lem_rhoB} that
\begin{align*}
y_1=\frac{x(1-{\rm th}^2(R/2))}{1-|x|^2{\rm th}^2(R/2)}-\frac{x(1-|x|^2){\rm th}(R/2)}{|x|(1-|x|^2{\rm th}^2(R/2))}=\frac{x(|x|-{\rm th}(R/2))}{|x|(1-|x|^2{\rm th}^2(R/2))}.    
\end{align*}
By Lemma \ref{lem_sCollinear}, we have
\begin{align*}
t=s_{\B^n}(x,y_1)=\frac{{\rm th}(R/2)(1-|x|^2)}{2(1-|x|{\rm th}(R/2))-|2|x|-{\rm th}(R/2)(1+|x|^2)|}. 
\end{align*}
Thus, the upper bound of $t_0$ follows, too.
\end{proof}

The next result follows directly from Corollary \ref{cor_rhoIncj} and, while all of its bounds might not be sharp, we know by Lemma \ref{lem_sRhoBounds} and its proof that the upper bound of $t_0$ and the lower bound of $R_1$ below are the best ones possible in the case $t\geq|x|$.

\begin{corollary}\label{cor_sRho}
For $x\in G=\B^n$ and $R>0$, $B_s(x,t_0)\subseteq B_\rho(x,R)\subseteq B_s(x,t_1)$ if
\begin{align*}
&0<t_0\leq\max\left\{\frac{(1+|x|){\rm sh}(R/2)}{2+(1+|x|){\rm sh}(R/2)},\frac{(1-|x|)(e^R-1)}{3+e^R-|x|(e^R-1)}\right\}
\quad\text{and}\\
&t_1\geq\left\{\frac{\sqrt{2}(1+|x|)(e^R-1)}{3+e^R+|x|(e^R-1)},\frac{(1+|x|)(e^R-1)}{4}\right\},
\end{align*}
and, for $x\in G=\B^n$ and $0<t<1$, $B_\rho(x,R_0)\subseteq B_s(x,t)\subseteq B_\rho(x,R_1)$ if
\begin{align*}
&0<R_0\leq\log\left(1+\frac{4k}{(1-k)(1+|x|)}\right)
\quad\text{with}\quad
k=\frac{t}{\min\{(1+t),\sqrt{2}\}}\quad\text{and}\\
&R_1\geq\min\left\{2{\rm arsh}\left(\frac{2t}{(1-t)(1+|x|)}\right),\log\left(1+\frac{4t}{(1-t)(1-|x|)}\right)\right\}. 
\end{align*}
\end{corollary}
\begin{proof}
Follows from Lemma \ref{lem_jSBinG} and Corollary \ref{cor_rhoIncj}.
\end{proof}

Computer test suggest that the following conjecture holds and, by comparing it to Lemma \ref{lem_sRhoBounds} and Corollary \ref{cor_sRho}, we see that, to prove this result, it is enough to show that the inclusion $B_\rho(x,R_0)\subseteq B_s(x,t)$ holds for $0<t<1$ and the inclusion $B_s(x,t)\subseteq B_\rho(x,R_1)$ for $0<t<|x|$.

\begin{conjecture}\label{conj_srhoB}
For $x\in G=\B^n$ and $R>0$, $B_s(x,t_0)\subseteq B_\rho(x,R)\subseteq B_s(x,t_1)$ if and only if
\begin{align*}
&0<t_0\leq\frac{{\rm th}(R/2)(1-|x|^2)}{2(1-|x|{\rm th}(R/2))-|2|x|-{\rm th}(R/2)(1+|x|^2)|}
\quad\text{and}\\
&t_1\geq\frac{(1+|x|)(e^R-1)}{3+e^R+|x|(e^R-1)},
\end{align*}
and, for $x\in G=\B^n$ and $0<t<1$, $B_\rho(x,R_0)\subseteq B_s(x,t)\subseteq B_\rho(x,R_1)$ if and only if
\begin{align*}
&0<R_0\leq\log\left(1+\frac{4t}{(1-t)(1+|x|)}\right)
\quad\text{and}\quad
R_1\geq2{\rm arsh}\left(\frac{2t}{\sqrt{l(t,|x|)}}\right),
\quad\text{where}\\
&l(t,|x|)=\max\{(1+t)(1+|x|)(1-3t+|x|(1+t)),\\
&\quad\quad\quad\quad\quad\quad\quad
(1-t)(1-|x|)(1+3t-|x|(1-t))\}.
\end{align*}
\end{conjecture}

\begin{figure}[ht]
    \centering
    \begin{tikzpicture}
    \node[inner sep=0pt] at (0,0)
    {\includegraphics[width=7cm,trim={1cm 2.5cm 1cm 3cm},clip]{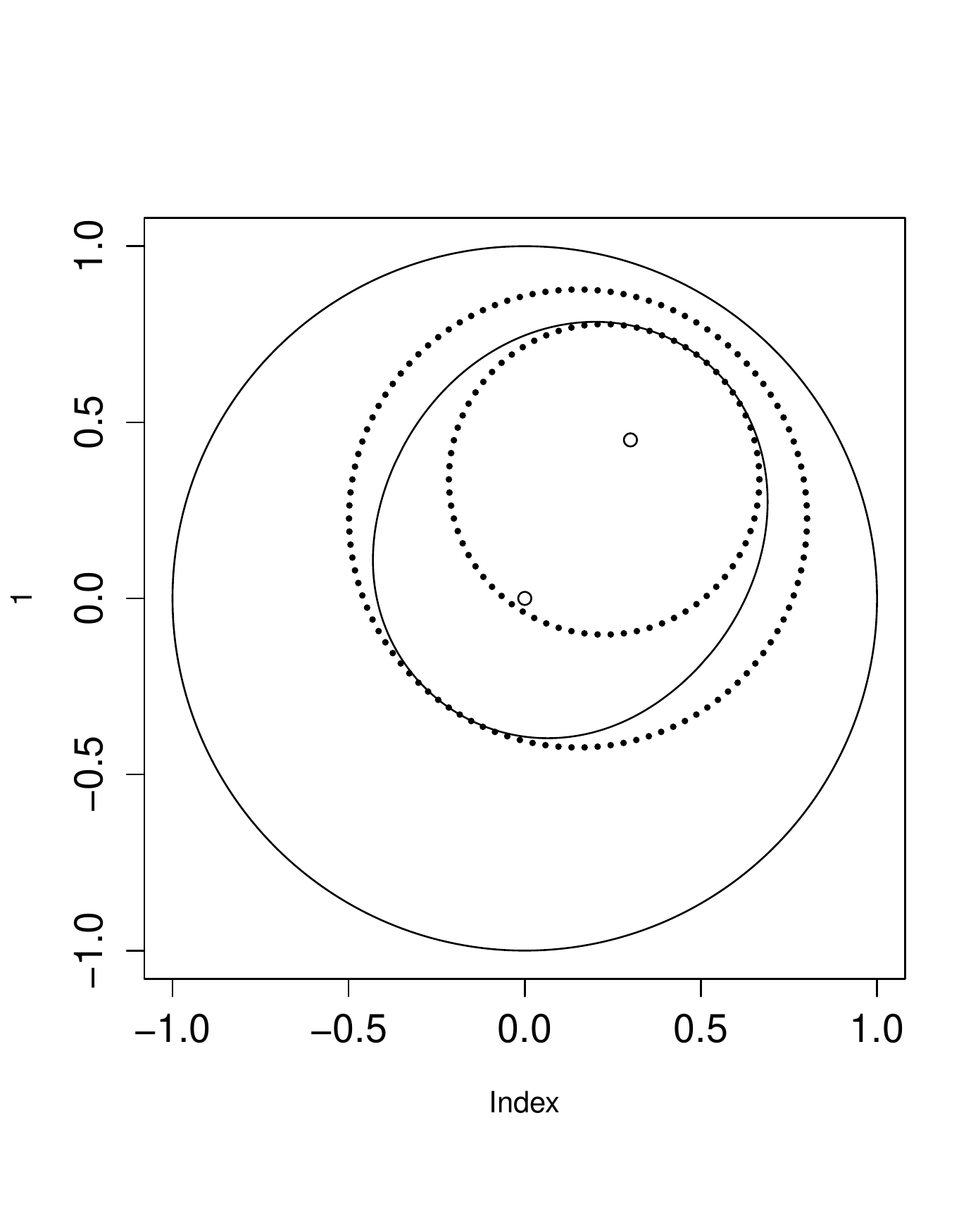}};
    \node[scale=1.3] at (0.3,0.6) {$0$};
    \node[scale=1.3] at (1.2,1.9) {$x$};
    \end{tikzpicture}
    \caption{The triangular ratio metric circle $S_s(x,t)$ for $G=\B^2$, $x=0.3+0.45i$ and $t=0.5$ as solid line and the hyperbolic circles $S_\rho(x,R_0)$ and $S_\rho(x,R_1)$ as dotted line, when $R_0$ is as large as possible and $R_1$ as small as possible within the bounds of Conjecture \ref{conj_srhoB}.}
    \label{fig5}
\end{figure}

Figure \ref{fig5} depicts the sharp inclusion of a triangular ratio metric balls and two hyperbolic metric balls found with Conjecture \ref{conj_srhoB}.

\addcontentsline{toc}{section}{References} 
\renewcommand{\refname}{References} 



\begin{thebibliography}{10}

\bibitem{chkv}{\sc
J. Chen, P. Hariri, R. Kl\'en and M. Vuorinen,}
Lipschitz conditions, triangular ratio metric, and quasiconformal maps.
\emph{Ann. Acad. Sci. Fenn. Math., 40} (2015), 683-709.

\bibitem{fhmv}{\sc
M. Fujimura, P. Hariri, M. Mocanu and M. Vuorinen,}
The Ptolemy–Alhazen Problem and Spherical Mirror
Reflection.
\emph{Comput. Methods and Funct. Theory, 19} (2019), 135-155.


\bibitem{hkv}{\sc
P. Hariri, R. Kl\'en and M. Vuorinen,}
\emph{Conformally Invariant Metrics and Quasiconformal Mappings.}
Springer, 2020.

\bibitem{hkvz}{\sc
P. Hariri, R. Kl\'en, M. Vuorinen and X. Zhang,}
Some Remarks on the Cassinian Metric.
\emph{Publ. Math. Debrecen, 90}, 3-4 (2017), 269-285.

\bibitem{hvz}{\sc
P. Hariri, M. Vuorinen and X. Zhang,}
Inequalities and Bilipschitz Conditions for Triangular Ratio Metric.
\emph{Rocky Mountain J. Math., 47}, 4 (2017), 1121-1148.

\bibitem{h}{\sc  
P. H\"ast\"o,} 
A new weighted metric, the relative metric I. \emph{J. Math. Anal. Appl. 274} (2002), 38-58.

\bibitem{h16}{\sc
S. Hokuni, R. Klén, Y. Li, and M. Vuorinen,}
Balls in the triangular ratio metric. Proceedings of an international conference, Complex Analysis and Dynamical Systems VI, Contemp. Math. Volume 667, 2016.

\bibitem{k12}{\sc
R. Klén and M. Vuorinen,}
Inclusion relations of hyperbolic type metric balls. \emph{Publ. Math. Debrecen, 81} (2012), 289-311

\bibitem{k13}{\sc
R. Klén and M. Vuorinen,}
Inclusion relations of hyperbolic type metric balls II. \emph{Publ. Math. Debrecen, 83/1-2} (2013), 21–42

\bibitem{ird}{\sc
O. Rainio,} 
Intrinsic metrics in ring domains.
\emph{Complex Anal. Synerg., 8}, 3 (2022). doi: 10.1007/s40627-022-00092-5

\bibitem{sch}{\sc
O. Rainio,} 
Intrinsic metrics under conformal and quasiregular mappings.
\emph{Publ. Math. Debrecen}, (to appear) (2022).

\bibitem{fss}{\sc
O. Rainio,}
Intrinsic quasi-metrics.
\emph{Bull. Malays. Math. Sci. Soc., 44}, 5 (2021), 2873-2891.

\bibitem{sinb}{\sc
O. Rainio and M. Vuorinen,}
Triangular ratio metric in the unit disk.
\emph{Complex Var. Elliptic Equ., 67}, 6 (2022), 1299-1325, doi: 10.1080/17476933.2020.1870452.

\bibitem{sqm}{\sc
O. Rainio and M. Vuorinen,}
Triangular Ratio Metric Under Quasiconformal Mappings In Sector Domains. \emph{Comput. Methods Func. Theory} (2022), doi: 10.1007/s40315-022-00447-3.

\end{thebibliography}
\end{document}